\documentclass[12pt]{amsart}
\usepackage{graphicx}
\usepackage{amsthm,mathrsfs}
\usepackage{amsfonts}
\usepackage{amsmath}
\usepackage{amssymb}
\newtheorem{lemma}{Lemma}
\newtheorem{theorem}{Theorem}

\usepackage{float}
\usepackage{url}
\usepackage{enumerate}

\newcommand {\ve} {\varepsilon}

\makeatletter\def\blfootnote{\xdef\@thefnmark{}\@footnotetext}\makeatother
\allowdisplaybreaks
\title[Additive Energy and Pair correlations]{\bf Additive Energy and the Hausdorff dimension of the exceptional set in metric pair correlation problems}

\author[Aistleitner]{Christoph Aistleitner}
\address{Institute of Analysis and Number Theory, TU Graz, Austria}
\email{aistleitner@math.tugraz.at}

\author[Larcher]{Gerhard Larcher}
\address{Institute of Financial Mathematics and applied Number Theory, University Linz, Austria}
\email{gerhard.larcher@jku.at}

\author[Lewko]{Mark Lewko.\\ With an appendix by Jean Bourgain}
\address{Department of Mathematics, University of California, Los Angeles, USA}
\email{mlewko@gmail.com}


\thanks{The first author is supported by the Austrian Science Fund (FWF) by an FWF Schr\"odinger scholarship, FWF project I1751-N26 and FWF START project Y-901. The first and the second author are supported by the FWF project F5507-N26, which is part of the Special Research Program \emph{Quasi-Monte Carlo Methods: Theory and Applications}}

\begin{document}

\dedicatory{Dedicated to the memory of Professor Steven A. Gaal (1924-2016)}

\begin{abstract}
For a sequence of integers $\{a(x)\}_{x \geq 1}$ we show that the distribution of the pair correlations of the fractional parts of $\{ \langle \alpha a(x) \rangle \}_{x \geq 1}$ is asymptotically Poissonian for almost all $\alpha$ if the additive energy of truncations of the sequence has a power savings improvement over the trivial estimate. Furthermore, we give an estimate for the Hausdorff dimension of the exceptional set as a function of the density of the sequence and the power savings in the energy estimate. A consequence of these results is that the Hausdorff dimension of the set of $\alpha$ such that $\{\langle \alpha x^d \rangle\}$ fails to have Poissonian pair correlation is at most $\frac{d+2}{d+3} < 1$. This strengthens a result of Rudnick and Sarnak which states that the exceptional set has zero Lebesgue measure. On the other hand, classical examples imply that the exceptional set has Hausdorff dimension at least $\frac{2}{d+1}$.\\

An appendix by Jean Bourgain was added after the first version of this paper was written. In this appendix two problems raised in the paper are solved.
\end{abstract}

\date{}
\maketitle

\section{Introduction} \label{sect_1}

We will be interested in the \emph{pair correlation statistics} of sequences of real numbers from the unit interval, which is defined as follows. Let $\theta_1, \dots, \theta_N \in [0,1]$, and let $\| \cdot \|$ denote the distance to the nearest integer. For every interval $[-s,s]$ we set
$$
R_2 \big( [-s,s],\{ \theta_n\}, N \big) = \frac{1}{N} ~\# \left\{ 1 \leq j \neq k \leq N: \left\|\theta_j - \theta_k \right\| \leq \frac{s}{N} \right\}.
$$
The subscript ``2'' of ``$R_2$'' refers to the fact that these are the \emph{pair} correlations, that is, the correlations of order 2, in contrast to triple correlations or correlations of even higher order. For a sequence of independent, $[0,1]$-uniformly distributed random variables $\theta_1, \theta_2, \dots$ for every $s \geq 0$ we have
$$
R_2 ( [-s,s],\{ \theta_n \}, N) \to 2s
$$
almost surely. If this asymptotic relation holds for a given sequence we say that the distribution of the pair correlations is asymptotically \emph{Poissonian}.\\

Of particular interest are the distributions of the pair correlations of sequences of the form $\{ \langle \alpha x^d \rangle \}_{x \geq 1}$,  where $\langle \cdot \rangle$ denotes the fractional part function. These occur as the distribution of the spacings of the energy levels of certain integrable systems. See the introduction of~\cite{rst} for a discussion of this connection. For $d \geq 2$, Rudnick and Sarnak~\cite{rst} proved that the distribution of the pair correlations is asymptotically Poissonian for almost all $\alpha$. The case $d=2$, which corresponds to the energy levels of the \emph{boxed oscillator}, has received particular attention, see for example~\cite{hp,mse,rsz,td}. Boca and Zaharescu~\cite{bzp} extended this result to show that the pair correlations of $\{\langle \alpha a(x) \rangle\}_{x \geq 1}$ are almost always Poissonian for any $a(x) \in \mathbb{Z}[x]$ of degree at least 2. In the case $d=1$ the situation is very different, and the distribution of the pair correlations is \emph{not} asymptotically Poissonian for any value of $\alpha$; this fact is related to the \emph{Three Gap Theorem} of S\'os~\cite{sos} and {\'S}wierczkowski~\cite{sw}.\\

Another case which has been intensively investigated is that of  $\left\{\langle \alpha a(x) \rangle \right\}_{x \geq 1}$ for $\{a(x)\}_{x \geq 1}$ being a \emph{lacunary} sequence, see for example~\cite{bpt,clz,rzt}. For brevity we will say a sequence $\{a(x)\}_{x\geq 1}$ has the \emph{metric pair correlation property} if $\left\{\langle \alpha a(x) \rangle \right\}_{x \geq 1}$ has asymptotically Poissonian pair correlations for almost all $\alpha$. \\

Metric results of this type are generally much easier to obtain than results for the corresponding problem for \emph{specific} values of $\alpha$. A similar phenomenon occurs in the theory of uniform distribution modulo one, where equidistribution results for $\{\langle \alpha a(x) \rangle\}_{x \geq 1}$ are relatively easy to obtain in the metric sense but can be extremely difficult for specific values of $\alpha$. Rudnick and Sarnak~\cite{rst} and Heath-Brown~\cite{hp} have conjectured that $\{ \langle \alpha x^d \rangle \}_{x \geq 1}$ has Poissonian pair correlations if $\alpha$ is a Diophantine number. Recall that a real number is said to be of type $\kappa$ if $\left| \alpha - \frac{p}{q}\right | \gg \frac{1}{q^{\kappa}}$ for all $p,q \in \mathbb{Z}$, and a number is said to be Diophantine if it is of type $\kappa$ for all $\kappa > 2$. It follows from Roth's theorem that all irrational algebraic numbers are Diophantine. Some form of a Diophantine condition is certainly required as it was observed in~\cite{rst} that $\{ \langle \alpha x^d \rangle \}_{x \geq 1}$ fails to have the metric pair correlation property if $\alpha$ is not of type $\kappa = d+1$.\\

While the Heath-Brown--Rudnick--Sarnak conjecture seems far from resolution, several authors have obtained results that suggest that the Diophantine condition might be able to be somewhat relaxed. In particular, Heath-Brown~\cite{hp} has shown that if $\alpha$ is of type $9/4$ then
$$R_2 \big( [-s,s],\{  \langle \alpha x^2 \rangle\}, N) \rightarrow 2s + \mathcal{O}(s^{7/8}).$$
In another direction, Truelsen~\cite{td} has formulated a strong conjecture regarding averaged divisor sums which implies that $\{  \langle \alpha x^2 \rangle\}$ has Poissonian pair correlations for all $\alpha$ of type $\kappa < 3$. We will offer some additional evidence in this direction. It is a consequence of Theorem~\ref{th1} below that the set of exceptional $\alpha$ for which $\{ \langle \alpha x^d \rangle \}_{x \geq 1}$ fails to have Poissonian pair correlations has Hausdorff dimension at most $\frac{d+2}{d+3} < 1$.  This can be thought of as a quantitative strengthening of Rudnick and Sarnak's result. On the other hand recall that the Jarn\'ick-Besicovitch Theorem states that the Hausdorff dimension of the set of real numbers which fail to be of type $\kappa>2$ is $\frac{2}{\kappa}$. Thus the examples of Rudnick and Sarnak mentioned above imply that the Hausdorff dimension of the exceptional set is at least $\frac{2}{d+1}$. Our Hausdorff dimension estimate is an application of a general result which, as we will discuss below, uses very limited information about these particular sequences. It seems likely that additional number theoretic input could be incorporated into our argument to obtain improved Hausdorff dimension estimates for these sequences but we will not pursue this here. \\

Our main result will link the pair correlation of an integer sequence $\{a(x)\}_{x>1}$ to the additive energy of its truncations. Recall that the additive energy of a set of real numbers $A$ is defined to be \begin{equation}\label{defAE}
E(A) := \sum_{a+b=c+d} 1,
\end{equation}
where the sum is extended over all quadruples $(a,b,c,d) \in A^4$. Trivially one has that $|A|^2 \leq E(A) \leq |A|^3$. Additive energy has been extensively studied in the combinatorics literature. We refer the reader to~\cite{tv} for a discussion of its properties and applications.
 To ease notation in the sequel when a sequence $A:=\{a(x)\}_{x>1}$ is fixed we will abbreviate $R_2 ( [-s,s],\alpha, N)$ for $R_2 ( [-s,s],\{ \langle \alpha a(x) \rangle \}, N)$. Furthermore we will let $A_N$ denote the first $N$ elements of $A$. Our main result states that if the truncations $A_N$ of a sequence $A$ satisfy $E(A_N) \ll N^{3-\ve}$ for some $\varepsilon >0$, then $\{ \langle \alpha a(x) \rangle \}_{x \geq 1}$ has Poissonian pair correlations for almost all $\alpha$. Moreover, if the sequence satisfies the growth condition $a(x) \ll x^d$ then we obtain an estimate for the Hausdorff dimension of the exceptional set in terms of $d$ and $\varepsilon$. In particular we obtain the following theorem.
\begin{theorem} \label{th1}
Let $\{a(x)\}_{x \geq 1}$ be a sequence of distinct integers, and suppose that there exists a fixed constant $\ve >0$ such that
\begin{equation} \label{anv}
E(A_N) \ll N^{3-\ve} \qquad \textrm{as $N \to \infty$}.
\end{equation}
Then for almost all $\alpha$ one has
\begin{equation} \label{poiss}
R_2 \big([-s,s],\alpha,N\big) \to 2 s \qquad \textrm{as $N \to \infty$}
\end{equation}
for all $s \geq 0$.
Moreover if $a(x) \ll x^d$, then the Hausdorff dimension of the set of $\alpha$ for which~\eqref{poiss} fails is at most
\begin{equation}
\frac{d+3-\ve}{d+3}.
\end{equation}
\end{theorem}

To the best of our knowledge, the first part of Theorem~\ref{th1} covers all sequences in the literature for which such pair correlations results have been obtained and significantly relaxes the following criteria obtained by Rudnick and Zaharescu~\cite{rza}:
\begin{quote}
\emph{Let $\{a(x)\}_{x \geq 1}$ be a sequence of distinct integers and suppose that there are at most $\mathcal{O} \left( M N^{2+\ve} \right)$ solutions to the equation}
$$
n_1 \big(a(x_1) - a(y_1) \big) = n_2 \big( a(x_2) - a(y_2) \big)
$$
\emph{with $1 \leq x_i \neq y_i \leq N$, and $1 \leq |n_i| \leq M$, $M \ll N^R$ for some $R > 0$, and all $\ve > 0$. Then for almost all $\alpha$ we have~\eqref{poiss}.}
\end{quote}
Even if one could chose $M=1$ in the above, this condition requires that $E(A_N) \ll N^{2 + \ve}$ for all $\ve$ as compared with the condition $E(A_N) \ll N^{3-\ve}$ for some $\ve >0$ in Theorem~\ref{th1}. Furthermore, the presence of the coefficients $n_1,n_2$ in this condition makes it difficult to verify for a specific sequence $\{a(x)\}_{x \geq 1}$. On the other hand additive energy estimates are known for a wide class of sequences. Applications of Theorem~\ref{th1} are given in Section 2 below.\\

The second part of Theorem~\ref{th1} should be compared to the corresponding results for the Hausdorff dimension of the exceptional set in equidistribution theory. Recall that a classical result of Weyl~\cite{weyl} states that every sequence $\{\langle \alpha a(x) \rangle\}_{x \geq 1}$ is equidistributed for almost all $\alpha$ for any sequence $\{a(x)\}_{x \geq 1}$ of distinct integers. Erd\H os and Taylor~\cite{eto} proved the finer result that for an integer sequence $\{a(x)\}_{x \geq 1}$ satisfying $a(x)=\mathcal{O}(x^d)$ the set of those $\alpha$ for which the fractional parts of $\{\alpha a(x)\}_{x \geq 1}$ are not asymptotically equidistributed has Hausdorff dimension at most $(d-1)/d$; this result is known to be optimal (see~\cite[Theorem 6]{ruzsa}). We note that Nair~\cite{nair} already obtained some results on the Hausdorff dimension of exceptional sets in pair correlations problems.  However his interest is in questions regarding the speed of convergence in~\eqref{poiss} and consequently his results are in a somewhat different direction than ours.\\

Our work and its analogy to the equidistribution setting raises two questions:\\

\textbf{Problem 1:} Is it possible for an increasing sequence of distinct integers $\{a(x)\}_{x \geq 1}$ which satisfies $E(A_N) = \Omega \left( N^{3} \right)$ to have Poissonian pair correlations for almost all $\alpha$?\footnote{Here the notation $E(A_N) = \Omega \left( N^{3} \right)$ means that $E(A_N) \geq c N^3$ for some positive constant $c$ and infinitely many $N$; in other words, we do \emph{not} have $E(A_N)=o(N^3)$.}\\

\textbf{Problem 2:} If $\{a(x)\}_{x \geq 1}$ is an increasing sequence of distinct integers, does $E(A_N) = o \left( N^{3} \right)$ imply that~\eqref{poiss} holds for almost all $\alpha$ for all $s \geq 0$? \\

Shortly after the first version of this paper was finished and was made available online, we were contacted by Jean Bourgain who could prove that the answer to both questions above is negative. We are very grateful to him for giving us permission to include his arguments as an appendix to this paper.

\section{Applications} \label{examples}
In this section we discuss applications of Theorem~\ref{th1}.\\

We start by observing that if a sequence satisfies the lacunary growth condition $a(x+1)/a(x) \geq c > 1, ~x \geq 1,$ then it is easy to see that $E(A_{N}) \ll N^{2}$. Thus Theorem~\ref{th1} immediately implies that $\{a(x)\}$ has the metric pair correlation property, recovering a result of~\cite{rza}.\\

Next we will discuss some number theoretic sequences which will require the following facts. Let $d(n)$ denote the number of divisors of $n$. Then for some universal constant $c_0>0$ and every $\varepsilon >0$ we have the so-called divisor bound
\begin{equation}\label{divisorbound}
d(n) \ll e^{c_0 \log(n)/\log \log (n)} \ll_{\varepsilon} n^{\varepsilon}.
\end{equation}
If $A \subset \mathbb{Z}$ is a finite set of integers and $r(n,A) = |\{a,b \in A : n=a-b\}|$ then, using~\eqref{defAE}, we have that
\begin{equation}\label{energyToR}
E(A) \leq |A|^{2} \max_{n \in \mathbb{Z}} r(n,A).
\end{equation}
In order to recover the fact that any polynomial sequences, say $a(x) \in \mathbb{Z}[x]$, of degree $d \geq 2$ satisfies the conclusion of Theorem~\ref{th1}, it suffices to show that $r(n, A)=|\{x, y \in A : a(x)-a(y) =n\}|$ satisfies $r(n,A) \ll_{\varepsilon} n^{\varepsilon}$ for all $\varepsilon >0$. To see this note that  $x^d-y^d = (x-y) (x^{d-1}+ x^{d-2}y + \ldots +xy^{d-2}+y^{d-1}) $. From this we see that $\left(P(x) - P(y)\right)= (x-y) Q(x,y)$ where $Q$ is a bivariate degree $d-1$ polynomial. It follows that the number of solutions to $n=P(x)-P(y)=(x-y)Q(x,y)$ is at most the number of divisor pairs $n=ab$ multiplied with the number of simultaneous solutions of $a=(x-y)$ and $b=Q(x,y)$. Substituting the first equation into the second and applying the factor theorem shows that the number of solutions is at most $d-1$, which proves the claim.  Using~\eqref{energyToR} and Theorem~\ref{th1}, this establishes that $\{a(x)\}_{x \geq 1}$ has the metric pair correlation property. It also follows from this and our main theorem that the Hausdorff dimension of the exceptional set is at most $\frac{d+2}{d+3}$ for a degree $d$ polynomial sequence. \\

The additive energy of various sequences has been studied extensively in the additive combinatorics literature and these results can be easily paired with our results to establish the metric pair correlation property for many new sequences.  For instance, Konyagin~\cite{Konyagin1} (see also~\cite{Garaev,Garaev2,GK}) has shown that a convex sequence $\{a(x)\}_{x \geq 1}$, by which we mean a sequence satisfying $a(x) - a(x-1) < a(x+1)-a(x)$ for all $x > 1$, will satisfy $E(A_N) \ll N^{5/2}$. This includes, for instance, Bochkarev's sequence $\{\lfloor e^{(\log x )^{\beta}} \rfloor\}$ for $\beta >1$. In the context of the Waring--Goldbach problem, Piatetski-Shapiro~\cite{piat} proved that the sequence $a(x) = \lfloor \beta x^{\alpha} \rfloor$ satisfies $E(A_N) \ll N^{4-\alpha}$ for $\beta>0$ and $1 < \alpha < 3/2$. Estimates for the additive energy of general sequences of the form $a(x) = \lfloor F(x) \rfloor$ for smooth $F$ are given in~\cite{GK}.\\

Furthermore, it follows from the proof that $E(A_N) \ll_{\varepsilon} N^{2+\varepsilon}$ for polynomial sequences $\{a(x)\}_{x \geq 1}$ that if $\{b(x)\}_{x \geq 1}$ is a subsequence of polynomial relative density then it has the metric pair correlation property as well. By this we mean a subsequence $\{b(n)\}$ that satisfies
$$\frac{ \left| \{ b(n) :  n \leq N  \} \right| }{\left| \{ a(n) :  n \leq N  \} \right|}  \gg N^{-1+\eta}$$ for some fixed $\eta>0$. This does not appear to follow from the previous methods used to analyze polynomial sequences. It seems likely that an arbitrary subsequence of a polynomial sequence has the metric pair correlation property, which would follow from the arguments above if one could show that $r(n,A)$ is uniformly bounded for a given polynomial sequences. This would follow\footnote{This was pointed out to the third author by Bobby Gizzard and Terry Tao.} from the Bombieri--Lang conjecture using the work of Caporaso, Harris, and Mazur~\cite{CHM} for polynomials of degree $5$ and higher. In another direction, using very different methods from additive combinatorics, Sanders~\cite{Sanders} has shown that for universal constants $c_1$ and $c_2$ one has
$$ E(A) \leq |A|^{3} e^{-c_1 \log^{c_2} |A|}  $$
for an arbitrary set $A$ of squares. If one could take $c_2 > 1/2$ then it would follow from our arguments that an arbitrary subsequence of the squares has the metric pair correlation property. Here the precise form of Bondarenko and Seip's GCD bound (\ref{gcdsum}) below would play an important role.

\section{Preliminary results} \label{prelim}

Similar to previous approaches~\cite{rst,rza} our method proceeds by estimating the expectation and variance of sums of the form $\sum_{x,y} f (\alpha a(x)-\alpha a(y))$. However, rather than using smooth test functions $f$ as in some of the previously cited papers we work directly with the indicator functions of the short intervals $[-s/N,s/N]$. Replacing these indicator functions by their respective Fourier series and using a combinatorial argument together with the orthogonality of the trigonometric system, we will reduce the problem of estimating the moments of such sums of dilated functions to a problem involving a certain GCD (greatest common divisors) sum. The role of GCD sums in metric number theory goes back at least to  Koksma~\cite{koks} (see also~\cite{koks2}), and they play a role in the context of the Duffin--Schaeffer conjecture in metric Diophantine approximation (see Dyer and Harman~\cite{dhs}) and in the theory of almost everywhere convergence of sums of dilated functions~\cite{abs,lrr}. We will need the following upper bound of Bondarenko and Seip~\cite{bsg} for such GCD sums. As usual $\exp(x):=e^x$.

\begin{lemma} \label{lemmagcd}
Let $m_1, \dots, m_M$ be distinct positive integers, and let $b_1, \dots, b_M$ be real numbers such that $b_1^2+ \dots + b_M^2 \leq 1$. Then there exists an absolute constant $\kappa$ such that
\begin{equation} \label{gcdsum}
\sum_{k,\ell=1}^M b_k b_\ell \frac{\gcd(m_k,m_\ell)}{\sqrt{m_k m_\ell}} \leq \exp \left(\frac{\kappa \sqrt{(\log M) \log \log \log M}}{\sqrt{\log \log M}}\right)
\end{equation}
(we assume that $M$ is so large that all the logarithmic terms are well-defined and positive).
\end{lemma}

Lemma~\ref{lemmagcd} is stated in a formulation without coefficients $b_1, \dots, b_M$ in Theorem 1 of~\cite{bsg}, and in a somewhat concealed form (formulated in terms of the largest eigenvalues of general GCD matrices) in Corollary 1 of~\cite{bsg}.\footnote{For the connection between GCD sums and eigenvalues of GCD matrices, see~\cite{absw}.} Coefficients can be added at the cost of an additional factor $\log N$ on the right-hand side; however, this additional factor is omitted in the statement of Lemma~\ref{lemmagcd} since it can be incorporated into the exponential term by taking a slightly larger value for $\kappa$. One should note that the quadratic form defined by the left-hand side of~\eqref{gcdsum} is positive definite; this fact can be established using methods from linear algebra (see~\cite[Example 3]{bl}) or using an interpretation of the left-hand side of~\eqref{gcdsum} as an inner product in an appropriate function space (see \cite[Lemma 1]{abs} and \cite{ls}).\\

The upper bound in Lemma~\ref{lemmagcd} is optimal (except for the value of the constant), as was recently shown by Bondarenko and Seip in~\cite{bsl}. However, for our proof of Theorem~\ref{th1} we do not actually need the full power of the result from~\cite{bsg}; the earlier estimates from~\cite{dhs} or~\cite{abs} would suffice as well. As a side note, the sum in~\eqref{gcdsum} is a GCD sum with parameter 1/2, while more generally these GCD sums contain the expression $(\gcd(m_k,m_\ell))^{2 \beta}/(m_k m_\ell)^\beta$ for some $\beta$, the most interesting cases being $\beta \in [1/2,1]$. Recent research has revealed an interesting connection between such GCD sums and the Riemann zeta function; see for example~\cite{al},~\cite{bsl} and~\cite{lrr}.\\

For the proof of the second part of Theorem~\ref{th1} we will also need the following properties of the Hausdorff dimension (\cite[Lemma 10 and Lemma 11]{nair}).\\

\begin{lemma} \label{lemmah}~
\begin{enumerate}[a)]
\item If $E_1 \subset \mathbb{R}$ is a set of Hausdorff dimension $\nu$, then there exists a compact set $E_2 \subset E_1$ such that $E_2$ also has Hausdorff dimension $\nu$.
\item Let $E \subset \mathbb{R}$ be a compact set whose Hausdorff dimension is greater than $\nu$. Then there exists a positive Borel measure $\mu$ supported on $E$ such that for every interval $[x,y]$ we have $\mu([x,y]) \leq (y-x)^\nu$.
\end{enumerate}
\end{lemma}
%

%
%
%

%

\section{The variance estimate} \label{sec:proof}

Throughout this section, the constants implied by the symbol ``$\ll$'' are independent of $N$ and $A = \{a(x)\}_{x \geq 1}$. Furthermore, if we assume that the value of $s$ is uniformly bounded away from zero and infinity, then the constants implied by ``$\ll$'' are also independent of $s$.\\

Let $s$ and $N$ be given, and assume that $2s \leq N$. For $\alpha \in \mathbb{R}$ we set
$$
\mathbf{I}_{s,N} (\alpha) := \left\{ \begin{array}{ll}  1 & \textrm{if $\|\alpha\| \leq \frac{s}{N}$} \\ 0 & \textrm{otherwise.} \end{array} \right.
$$
In other words, $\mathbf{I}_{s,N}$ is the indicator function of the interval $[-s/N,s/N]$, extended with period 1. With this notation we have that
\begin{equation} \label{r2rep}
R_2 \big([-s,s],\alpha,N\big) = \frac{1}{N} \sum_{\substack{1 \leq x,y \leq N,\\x \neq y}} \mathbf{I}_{s,N} \big(\alpha (a(x) -a(y))\big),
\end{equation}
and
\begin{equation} \label{r2int}
\int_0^1 R_2 \big([-s,s],\alpha,N\big) ~d\alpha = \frac{2(N-1)s}{N}.
\end{equation}
The main technical lemma is the following which may be seen as a pair correlation version of the Erd\H{o}s--Tur\'an inequality:
\begin{lemma}\label{lem:L2est}
We have
$$\int_0^1 \left( R_2 \big([-s,s],\alpha,N\big)  -  \frac{2(N-1)s}{N} \right)^2 ~d \alpha$$
$$\ll E(A_N) N^{-3} \exp \left(\frac{\hat{\kappa} \sqrt{\log N \log \log \log N}}{\sqrt{\log \log N}}\right),$$
where $\hat{\kappa}$ is an absolute constant. (We assume that $N$ is so large that all logarithms are well-defined and positive.)
\end{lemma}

\begin{proof}
We may expand $\mathbf{I}_{s,N} (\alpha)$ in a Fourier series as
$$
\mathbf{I}_{s,N} (\alpha) \sim \sum_{\substack{n \in \mathbb{Z} \\ n \neq 0 } } c_n e(n \alpha),
$$
with $e(\alpha)=e^{2 \pi i \alpha}$ and
$$c_n= \frac{\sin(2\pi n s N^{-1})}{\pi n}.$$
Combining this with the trivial estimate $|c_n| \leq \int_{-sN^{-1}}^{sN^{-1}} |e(-n \alpha)| ~d\alpha \leq \frac{2s}{N}$ gives
\begin{equation} \label{fours}
|c_n| \leq \min \left(\frac{2s}{N}, \frac{1}{|n|} \right).
\end{equation}
Using~\eqref{r2rep} and expanding $\mathbf{I}_{s,N}$ in its Fourier series gives us
\begin{eqnarray}
& & \int_0^1 \left( R_2 \big([-s,s],\alpha,N\big)  -  \frac{2(N-1)s}{N} \right)^2 ~d \alpha \label{line*}\\
& = & \frac{1}{N^2} \int_0^1 \left( \sum_{\substack{1 \leq x,y \leq N,\\x \neq y}} \quad \sum_{\substack{n \in \mathbb{Z} \\ n \neq 0 }} c_n e\left( n \alpha (a(x) - a(y)) \right) \right)^2 ~d \alpha. \label{line1}
\end{eqnarray}
Expanding the square and integrating, the expression in line~\eqref{line1} is equal to
\begin{eqnarray*} & & \frac{1}{N^2} ~\sum_{\substack{1 \leq x_1, x_2, y_1, y_2 \leq N,\\x_1 \neq y_1,~x_2 \neq y_2}} ~\sum_{n_1,n_2 \in \mathbb{Z} \backslash \{0\}} 
c_{n_1} c_{n_2} ~\times \\
& &  \times ~ \int_{0}^{1} e \Big( \alpha \big(n_1 \left( a(x_1)-a(y_1)\right) - n_2 \left( a(x_2)-a(y_2)\right) \big)  \Big) d\alpha.
\end{eqnarray*}
Defining $R_N(v)$, the number of representations of an integer $v$, by
$$
R_N(v) := \# \left\{ (x,y) \in \{1, \dots, N\}^2, ~x \neq y:~~a(x)-a(y) = v \right\},
$$
the quantity above is
\begin{eqnarray*}
& & \frac{1}{N^2} ~\sum_{v,w \in \mathbb{Z} \backslash \{0\}} ~\sum_{\substack{n_1,n_2 \in \mathbb{Z} \backslash \{0\}, \\ n_1v = n_2 w}} R_N(v) R_N(w) c_{n_1} c_{n_2} \\
& = & \frac{1}{N^2} ~\sum_{v,w \in \mathbb{Z} \backslash \{0\}} R_N(v) R_N(w) ~\sum_{\substack{n_1,n_2 \in \mathbb{Z} \backslash \{0\}, \\ n_1v = n_2 w}} c_{n_1} c_{n_2}.
\end{eqnarray*}
Lemma~\ref{lem:L2est} now follows from Lemma~\ref{lemmagcd} once we show that
\begin{eqnarray}\label{cntoGCD}\sum_{\substack{n_1,n_2 \in \mathbb{Z} \\ n_1,n_2 \neq 0, \\ n_1v = n_2 w}} |c_{n_1} c_{n_2}| \ll (\log N) \frac{s}{N} \frac{\gcd(v,w)}{\sqrt{|vw|}}, \qquad v,w \neq 0,
\end{eqnarray}
if we assume that $\hat{\kappa} > \kappa$, where $\kappa$ is the absolute constant from Lemma~\ref{lemmagcd}.\\

In the sequel we will assume that $v,w \neq 0$. Note that $n_1 v = n_2 w$ if and only if $n_1 = \frac{h w}{\gcd(v,w)}$ and $n_2 = \frac{h v}{\gcd(v,w)}$ for some integer $h$. Using this and~\eqref{fours} we can record the following estimates on the quantity $|c_{n_1} c_{n_2}|$. For values of $|h| \leq \frac{N \gcd(v,w)}{s \max(|v|,|w|)}$ the following inequality is efficient
\begin{equation}\label{eq1}
|c_{n_1} c_{n_2}| \leq \frac{4s^2}{N^2}.
\end{equation}
If $\frac{N \gcd(v,w)}{s \max(|v|,|w|)} \leq |h| \leq \frac{N \gcd(v,w)}{s \min(|v|,|w|)}$, then
\begin{equation}\label{eq2}
|c_{n_1} c_{n_2}| \leq \frac{2s}{N} \frac{\gcd(v,w)}{h \max(|v|,|w|)}.
\end{equation}
Finally, if $|h| \geq \frac{N \gcd(v,w)}{s \min(|v|,|w|)}$, then
\begin{equation}\label{eq3}
|c_{n_1} c_{n_2}| \leq \frac{\gcd(v,w)^2}{h^2 |v w|}.
\end{equation}

As a consequence, for fixed $v,w$ we have
\begin{eqnarray*}
& & \sum_{\substack{n_1,n_2 \in \mathbb{Z} \backslash \{0\}, \\ n_1v = n_2 w}} |c_{n_1} c_{n_2}|\\
& \leq & \frac{2 N \gcd(v,w)}{s \max(|v|,|w|)} \frac{4s^2}{N^2}  + ~2 \sum_{\frac{N \gcd(v,w)}{s \max(|v|,|w|)} \leq h \leq \frac{N \gcd(v,w)}{s \min(|v|,|w|)}} \frac{2s}{N} \frac{\gcd(v,w)}{h \max(|v|,|w|)} \\
& & + ~2 \sum_{h \geq \frac{N \gcd(v,w)}{s \min(|v|,|w|)}} \frac{\gcd(v,w)^2}{h^2 |v w|} \\
& \ll & (\log N) \frac{s}{N} \frac{\gcd(v,w)}{\max(|v|,|w|)} +  \frac{s \min(|v|,|w|)}{N \gcd(v,w)} \frac{\gcd(v,w)^2}{|v w|}
\end{eqnarray*}
\begin{equation}
 \ll (\log N) \frac{s}{N}\frac{ \gcd(v,w)}{ \sqrt{|vw|}}.
\end{equation}
This completes the proof of Lemma~\ref{lem:L2est}.
\end{proof}

\section{Proof of Theorem~\ref{th1}} \label{sec_proof_th1}
In order to obtain the desired asymptotic result for almost all $\alpha$ we can use standard methods, such as the one used in Sections 3.2 and 3.3 of~\cite{rza}. 

Let $\gamma$ be a real number sufficiently large such that $\gamma \varepsilon > 1$. For $M \geq 1$, let
$$
N_M = \left\lceil M^\gamma \right\rceil.
$$
Let $s > 0$ be fixed. Then a combination of the variance estimate from the previous section, Chebyshev's inequality and the first Borel-Cantelli lemma easily implies that for almost all $\alpha$ we have
$$
\lim_{M \to \infty} R_2 \big([-s,s],\alpha,N_M\big) \to 2s \qquad \textrm{as $M \to \infty$},
$$
where the exceptional set depends on $s$. Thus we have the desired convergence behavior along a subsequence of $\mathbb{N}$. Recall that, as noted at the beginning of the previous section, the error terms in the crucial variance estimate hold uniformly in $s$ if we assume that $s$ is bounded away from 0 and $\infty$. Thus we can also prove that
$$
R_2\left(\left[-s \frac{N_{M}}{N_{M+1}}, s \frac{N_{M}}{N_{M+1}} \right],\alpha,N_M \right) \to 2s
$$
and
$$
R_2 \left( \left[-s\frac{N_{M+1}}{N_M}, s\frac{N_{M+1}}{N_M} \right],\alpha,N_{M+1} \right) \to 2s
$$
as $M \to \infty$. For $N$ satisfying$N_M \leq N \leq N_{M+1}$ we have, by definition,
\begin{eqnarray*}
& & N_M R_2 \left(\left[-\frac{s N_{M}}{N_{M+1}}, \frac{s N_{M}}{N_{M+1}} \right],\alpha,N_M \right) \\
& \leq & N R_2 \big([-s,s],\alpha,N \big) \\
& \leq & N_{M+1} R_2 \left( \left[-\frac{s N_{M+1}}{N_M}, \frac{s N_{M+1}}{N_M} \right],\alpha,N_{M+1} \right).
\end{eqnarray*}
Thus by $N_{M+1}/N_M \to 1$ we have
$$
\lim_{N \to \infty} R_2 \big([-s,s],\alpha,N \big) \to 2s \qquad \textrm{as $N \to \infty$},
$$
except for a set of $\alpha$'s which has Lebesgue measure zero (and which depends on $s$). Finally, to obtain a result for \emph{all} possible values of $s$ rather than a single fixed value of $s$, we repeat the whole argument for all $s$ from a dense subset of $\mathbb{R}^+$. Then the total exceptional set of $\alpha$'s still has measure zero, and we obtain the desired result.

\section{Proof of Hausdorff Estimate} \label{sec:proof2}

The proof will proceed by contradiction. Let $E$ denote the exceptional set from the statement of the theorem. Then by Lemma~\ref{lemmah} there exist a positive measure $\mu$ supported on $E$ and a number $\nu$ satisfying
\begin{equation} \label{nusize}
\nu > \frac{d+3-\ve}{d+3}
\end{equation}
such that $\mu([x,y]) \leq (y-x)^\nu$ for all intervals $[x,y] \subset [0,1]$. In the sequel we will use estimates similar to those in Section~\ref{sec:proof} to prove that for $\mu$-almost all $\alpha \in E$ the distribution of the pair correlations of $\{\langle \alpha a(x) \rangle\}_{x \geq 1}$ is asymptotically Poissonian. This clearly is a contradiction, which proves the second part of Theorem~\ref{th1}.\\


Using the same notation as in Section~\ref{sec:proof}, by applying Minkowski's inequality to~\eqref{line*} and~\eqref{line1} we have
\begin{eqnarray} 
& & \left( \int_0^1 \left( R_2 \big([-s,s],\alpha,N\big)  -  \frac{2(N-1)s}{N} \right)^2 ~d\mu(\alpha) \right)^{1/2} \label{gmint}\\
& \leq & \sum_{m=0}^\infty \left( \frac{1}{N^2} \int_0^1 \left| \sum_{\substack{1 \leq x,y \leq N,\\x \neq y}}  \sum\mathop{}_{\mkern-5mu m} ~c_n e( n \alpha (a(x) - a(y))) \right|^2 d\mu(\alpha) \right)^{1/2}, \nonumber
\end{eqnarray}
where (here and in the subsequent formula) the sum $\sum\mathop{}_{\mkern-5mu m}$ is extended over those integers $n$ which satisfy $(2^{m} - 1) N < |n| \leq (2^{m+1}-1)N$. To denote the integrands appearing above we define the functions
$$g_m(\alpha) = \left| \sum_{\substack{1 \leq x,y \leq N,\\x \neq y}} ~\sum\mathop{}_{\mkern-5mu m}~ c_n e( n \alpha (a(x) - a(y))) \right|^2, \qquad m \geq 0. $$

Proceeding as in the proof of Lemma~\ref{lem:L2est} we can show that there is an (arbitrarily small) constant $\eta_1>0$ such that
\begin{equation} \label{integrand}
\frac{1}{N^2} \int_0^1 g_m(\alpha) ~d\alpha \ll N^{-\ve+\eta_1} 2^{-m}.
\end{equation}
Indeed when $m=0$ the argument carries over verbatim. For $m \geq 1$ it suffices to show that
 
\begin{eqnarray}\label{cntoGCDm}\sum_{\substack{(2^{m} - 1) N < |n_1|,|n_2| \leq (2^{m+1}-1)N, \\ n_1v = n_2 w}} |c_{n_1} c_{n_2}| \ll 2^{-m} \frac{\gcd(v,w)}{\sqrt{|vw|}}, \qquad v,w \neq 0.
\end{eqnarray}

Recalling $n_1 v = n_2 w$ if and only if $n_1 = \frac{h w}{\gcd(v,w)}$ and $n_2 = \frac{h v}{\gcd(v,w)}$ for some integer $h$, using the restrictions on $n_1$ and $n_2$ we have that
$$h \geq (2^{m}-1) N \frac{\gcd(v,w)}{\min(v,w)}.$$
Thus one can proceed with the argument in case three above. More precisely, using inequality~\eqref{eq3}, one has
\begin{eqnarray*}
\sum_{\substack{(2^{m} - 1) N < |n_1|,|n_2| \leq (2^{m+1}-1)N, \\ n_1v = n_2 w}} |c_{n_1} c_{n_2}| 
& \ll & \sum_{h \geq (2^m-1) N \frac{\gcd(v,w)}{\min(v,w)} } \frac{\gcd(v,w)^2}{h^2 |vw|} \\
& \ll & \frac{\min(v,w)}{(2^m-1)N \gcd(v,w)} \frac{\gcd(v,w)^2}{|vw|} \\
& \ll & 2^{-m}\frac{\gcd(v,w)}{\sqrt{|vw|}},
\end{eqnarray*}
which give~\eqref{cntoGCDm}.
Note that in equation~\eqref{integrand}, the integration is carried out with respect to the Lebesgue measure, as in~\eqref{line1}. Thus to obtain an upper bound for~\eqref{gmint} we have to transform the estimate for $\int g_m(\alpha) d\alpha$ into an estimate for $\int g_m (\alpha) d\mu(\alpha)$.\\

By the growth condition on $a(x)$ and by~\eqref{fours} we have
\begin{equation} \label{deriv}
\|g_m'\|_\infty \leq K 2^m N^{d+5}
\end{equation}
for some universal positive constant $K$. Note also that
\begin{equation} \label{gsizemax}
\|g_m\|_\infty \ll N^4
\end{equation}
as a consequence of~\eqref{fours}. Let $R = 6 \log N$. Then for sufficiently large $N$ by~\eqref{gsizemax} we have $2^R \geq \|g_m\|_\infty$. By~\eqref{nusize} we have $\nu > \frac{d+3-\ve}{d+3}$, so there exist $\eta_2,\eta_3>0$ such that $d+3-\ve+\eta_1+\eta_2 < \nu (d+3+\eta_2)$. Take also $\eta_3=\eta_3(m)=2^ {-m}$.\\

Let $r \in \{0, \dots, R\}$. We split $[0,1]$ into $\lceil K 2^m N^{d+3+\eta_2} \rceil$ equally spaced subintervals, and let $B_{m,r}$ denote the collection of all those intervals which contain a point $\alpha$ where
$$
g_m(\alpha) \in \left[2^r N^{2-\eta_2} 2^{-\eta_3 m},2^{r+1} N^{2-\eta_2} 2^{-\eta_3 m} \right].
$$
Note that by~\eqref{deriv} and the mean value theorem for any other point $\hat{\alpha}$ in the same subinterval of $B_{m,r}$ as $\alpha$ we have
\begin{equation} \label{gsize}
|g_m(\hat{\alpha}) - g_m(\alpha) | \ll K 2^m N^{d+5} K^{-1} 2^ {-m-2} N^{-d-3-\eta_2} = \frac{N^{2-\eta_2}}{4}.
\end{equation}
Since $r - \eta_3 m \geq -1$, this shows that
$$
2^{r-1} N^{2-\eta_2} 2^{-\eta_3 m} \leq g_m \leq 2^{r+2} N^{2-\eta_2} 2^{-\eta_3 m}
$$
in the whole subinterval of $B_{m,r}$ containing $\alpha$. Thus, by~\eqref{integrand}, the number of such subintervals of $B_{m,r}$ is $\ll 2^{-r} N^{d+3-\ve+\eta_1+2\eta_2} 2^{\eta_3 m}$. Note that by our choice of $R$  and by~\eqref{gsize} we have
$$
\int_0^1 g_m(\alpha) d\mu(\alpha) \ll \sum_{r=0}^R 2^{r} N^{2-\eta_2} 2^{-\eta_3 m} \mu(B_{m,r})$$ 
$$ + N^{2-\eta_2} 2^{-\eta_3 m} \mu \left([0,1] \backslash \bigcup_{r=0}^R B_{m,r}\right).$$
Furthermore, using the properties of $\mu$ given in Lemma~\ref{lemmah}, we have
$$
\mu (B_{m,r}) \ll 2^{-r} N^{d+3-\ve+\eta_1+2\eta_2} 2^{\eta_3 m} \left(2^m N^{d+3+\eta_2}\right)^{-\nu}.
$$
Thus we have
$$ \frac{1}{N^2} \int_0^1 g_m(\alpha) ~d\mu(\alpha) \ll $$
\begin{eqnarray*}
 \sum_{r=0}^R N^{d+3-\ve+\eta_1+\eta_2} \left(2^m N^{d+3} N^{\eta_2}\right)^{-\nu}  + N^{-\eta_2} 2^{-\eta_3 m}.
\end{eqnarray*}
The quantity on the right-hand side of this equation is summable in $m$. Due to the small choice of the constants $\eta_1,\eta_2>0$ there exists $\eta_4 = \frac{1}{2} \min \big( \eta_2,\nu(d+3+\eta_2) - (d+3-\ve+\eta_1+\eta_2)\big)>0$ such that 
\begin{equation} \label{eta5}
\int_0^1 \left( R_2 \big([-s,s],\alpha,N\big)  -  \frac{2(N-1)s}{N} \right)^2 ~d\mu(\alpha) \ll N^{-\eta_4} \log N.
\end{equation}
Using~\eqref{eta5} in place of Lemma~\ref{lem:L2est}, we can proceed as in the proof of Theorem~\ref{th1} to show that the asymptotic distribution of the pair correlations is Poissonian for almost all $\alpha$ with respect to $\mu$. However, this is in contradiction with the fact that $\mu$ is supported on a set where the distribution of pair correlations is \emph{not} asymptotically Poissonian and establishes the Hausdorff dimension estimate.

\section{Appendix (by Jean Bourgain)}

The first problem stated at the end of the introduction asks whether  a sequence which has additive energy of maximal order may have the metric pair correlation property. To show that the answer is negative, we start by recalling the Balog--Szem\'eredi--Gowers lemma (see for example~\cite[Section 2.5]{tv}). We write $B-B$ for the difference set $\{b_1-b_2:~b_1,b_2 \in B\}$ of a set $B$, and $|B|$ for the cardinality of $B$.

\begin{lemma} \label{lem:BSG} Let $A \subset \mathbb{Z}$ be a finite set of integers. For any $c >0$ there exist $c_1, c_2 >0$ depending only on $c$ such that the following holds. If $E(A) \geq c |A|^3$,
then there is a subset $B \subset A$ such that
\begin{enumerate}[(i)]
\item \quad $|B| \geq c_1 |A|,$
\item \quad $|B-B | \leq c_2 |A|. $
\end{enumerate}
\end{lemma}
Next, observe that for a set of nonzero integers $S$ and any $\varepsilon > 0$ we have
\begin{eqnarray*} 
& & \textup{mes} \left( \left\{  \alpha \in [0,1] :~ \min_{n \in S} \| \left<n\alpha \right> \| < \frac{\varepsilon}{|S|} \right\} \right) \\
& \leq & \sum_{n \in S} \textup{mes} \left( \left\{ \alpha \in [0,1] :~ \| \left<n\alpha \right> \| < \frac{\varepsilon}{|S|}  \right\} \right) \leq 2 \varepsilon,
\end{eqnarray*}
where $\textup{mes}$ denotes Lebesgue measure. This immediately implies the following lemma.
\begin{lemma}\label{lem:meas}Let $B \subset \mathbb{Z}$ be a finite set of integers. Then for every $\varepsilon \in (0,1)$ we have
$$ \textup{mes}  \left( \left\{ \alpha \in [0,1] :~ \min_{\substack{m,n \in B \\ m \neq n}} \| \left<m \alpha \right> -  \left<n \alpha \right> \| < \frac{\varepsilon}{|B-B|}  \right\} \right) \leq 2 \varepsilon. $$
\end{lemma}

Now let an infinite sequence $(a(x))_{x \geq 1}$ be given, and assume that there exists a constant $c>0$ such that $E(A_N)>c N^ 3$ for infinitely many $N$. Let $N$ be an index for which this is true, and let $c_1,c_2$ be the constants and $B_N$ be the corresponding set as given by Lemma~\ref{lem:BSG}. Set $\varepsilon = \frac{1}{10} c_1^2$. It follows from Lemma~\ref{lem:meas} that there exists a set $\Omega_{\varepsilon} \subset [0,1]$ with $\textup{mes} (\Omega_{\varepsilon}) \leq 2 \varepsilon$ such that for all $m \neq n \in B_N$ we have

\begin{equation}\label{eq:diff}
\| \left< m \alpha \right> - \left< n \alpha \right> \| \geq \frac{\varepsilon}{c_2 N}
\end{equation}
for all $\alpha \not\in \Omega_{\varepsilon}$. Taking $s = \frac{\varepsilon}{2c_2}$ and setting 
$$
\mathcal{D}_N = \left\{ (m,n) \in (A_N \times A_N) \setminus (B_N \times B_N), ~m \neq n\right\},
$$
by~\eqref{eq:diff} it follows that for $\alpha \notin \Omega_\varepsilon$ we have 
\begin{eqnarray*}
R_{2}\left([-s,s], \alpha, N\right) & = & \frac{1}{N} \left| \left\{ (m, n) \in \mathcal{D}_N:~ \| \left< m \alpha \right> - \left< n \alpha \right> \| \leq \frac{s}{N} \right\}  \right|.
\end{eqnarray*}
By Lemma~\ref{lem:BSG} we have
\begin{eqnarray*}
& & \int_{0}^{1}\frac{1}{N} \left| \left\{ (m, n) \in \mathcal{D}_N :~ \| \left< m \alpha \right> - \left< n \alpha \right> \| \leq \frac{s}{N} \right\}  \right| d \alpha \\ 
& = & \frac{1}{N} \left( N^2 - |B_N|^2 \right) \frac{2s}{N} \\
& \leq & 2 \left(1 - c_1^2 \right) s.
\end{eqnarray*}
Thus there exists $\Omega' \subset [0,1]$ with $\textup{mes} (\Omega') \geq 1 - \frac{1-c_1^2}{1-c_1^2 /2} \geq \frac{c_1^2}{2}$ such that for $\alpha \in \Omega' \subset [0,1]$  we have
\begin{eqnarray*}
\frac{1}{N} \left| \left\{ (m, n) \in \mathcal{D}_N : \| \left< m \alpha \right> - \left< n \alpha \right> \| \leq \frac{s}{N} \right\}  \right| 
& \leq & 2 \left( 1 - \frac{c_1^2}{2}\right)s. 
\end{eqnarray*}
Therefore, for $\alpha \in \left(\Omega' \setminus \Omega_\varepsilon \right)$ we have
\begin{equation} \label{b7}
R_2([-s,s],\alpha,N) \leq 2 \left( 1 - \frac{c_1^2}{2}\right) s.
\end{equation}
From the choice of $\varepsilon$ we have
$$
\textup{mes} \left(\Omega' \setminus \Omega_\varepsilon \right) \geq \frac{c_1^2}{2} - 2 \varepsilon > \frac{c_1^2}{4}.
$$
Consequently, for a set of measure at least $\frac{c_1^2}{4}$ inequality~\eqref{b7} holds for infinitely many $N$. Thus the answer to the question in the first problem at the end of the introduction is negative.\\

Now we come to the second problem, which asks whether it is possible to relax the condition $E(A_N) \ll N^{3 - \varepsilon}$ from the statement of Theorem~\ref{th1} to $E(A_N) = o(N^3)$ (which would then be optimal, in light of the negative answer to the first problem). To construct a counterexample, let $K_N$ be a very slowly growing integer-valued function of $N$. Let $A_N$ denote a random subset of $\{K_N N +1,K_N N + 2,\ldots,2 K_N N \}$, which is obtained by setting $A_N= \{K_N N + n: ~1 \leq n \leq K_N N \text{ and }~\xi_n^{(N)}(\omega)=1\}$, where $\xi_1^{(N)}, \dots, \xi_{K_N N}^{(N)}$ are independent, $\{0,1\}$-valued random variables with mean $1/K_N$. 

\begin{lemma}\label{lem:RanEner} With positive probability all the following three properties hold.
\begin{enumerate}[(i)]
\item \quad For all $k \in \mathbb{Z}\setminus \{0\}$ we have $|A_N \cap \left(A_N +k \right)| \leq \frac{2N}{K_N},$ 
\item \quad For all $k \in \mathbb{Z}\setminus \{0\},~|k| < \frac {K_N N}{10}$, we have $|A_N \cap \left(A_N +k \right)| > \frac{N}{2 K_N}.$ 
\item \quad We have $N/2 \leq |A_N| \leq 2 N$.
\end{enumerate}
\end{lemma}
\begin{proof}
For $k \neq 0$ we have $\mathbb{E} \left(\xi_n^{(N)} \xi_{n-k}^{(N)}\right)=K_N^{-2}$. By construction, for $k \neq 0$, we have
\begin{equation} \label{autocorr}
|A_N \cap (A_N + k)| = \sum_{\max (1,1+k) \leq n \leq \min (K_N N,K_N N + k)} \xi_n^{(N)} \xi_{n-k}^{(N)}.
\end{equation}
The expression on the right-hand side of~\eqref{autocorr} is a random variable whose expected value is
\begin{equation} \label{b4}
\frac{|n:~\max (1,1+k) \leq n \leq \min (K_N N,K_N N + k)|}{K_N^ 2} \leq \frac{N}{K_N}.
\end{equation}
One can use the concentration of measures phenomenon and large deviation inequalities to prove that the probability of observing a value of~\eqref{autocorr} which is far from its mean is very small. More precisely, the quantity on the right-hand side of~\eqref{autocorr} is called the \emph{aperiodic autocorrelation at shift $(-k)$} of the sequence $\xi_1^{(N)}, \dots, \xi_{K_N N}^{(N)}$, and the supremum of its modulus (taken over all admissible values of $k$) is called the \emph{peak sidelobe level}. These are notions that have been studied intensively for random binary sequences. It is known that the distribution of the peak sidelobe level is strongly concentrated, which follows roughly speaking from the fact that when the total index set is cut into several pieces, then only random variables from the same or from neighboring segments are dependent, while all others are mutually independent. A detailed proof of assertion (i) of the lemma could be given using the methods from~\cite{alon,schmidt}.\\

Furthermore, if $|k| < \frac{K_N N}{10}$, then the left-hand side of~\eqref{b4} is at least $\frac{3N}{4 K_N}$, and again distributional considerations imply assertion (ii) of the lemma with large probability. Property (iii) also is true with large probability, again as a consequence of large deviation bounds.
\end{proof}

In the sequel, assume that $A_N$ denotes a specific realization of a sequence as described above, which satisfies all the three assertions of Lemma~\ref{lem:RanEner}. It follows from (i) of Lemma~\ref{lem:RanEner} that 
\begin{eqnarray*}
E(A_N) & = & \sum_{|k| \leq K_N N} |A_N \cap \left(A_N + k \right)|^2 
\\
& \leq & 2 K_N N \left(\frac{2 N}{K_N}\right)^2 \\
& = & \frac{8 N^3}{K_N} = o(N^{3}). 
\end{eqnarray*}
By a similar reasoning we may actually assume that a corresponding estimate for the additive energy holds uniformly along all initial segments of $A_N$. Next, for any $\alpha \in [0,1]$, using assertions (ii) and (iii) of Lemma~\ref{lem:RanEner}, for the pair correlations of this sequence we have
\begin{eqnarray}
R_{2}\left([-1,1], \alpha, A_N \right) & = & \frac{1}{|A_N|} \sum_{k \neq 0 }|A_N \cap \left(A_N + k \right)|~ \mathbf{1}_{\left(\|k \alpha\| \leq \frac{1}{|A_N|} \right)} \nonumber\\
& \geq & \frac{1}{2N} \frac{N}{2 K_N}  \sum_{0 < |k| < \frac{k_N N}{10}}~\mathbf{1}_{\left(\|k \alpha\| \leq \frac{1}{2N} \right)}\nonumber\\
& = & \frac{1}{2 K_N} \sum_{0 < k < \frac{K_N N}{10}} \mathbf{1}_{\left(\|k \alpha\| \leq \frac{1}{2N}\right)}, \label{b26}
\end{eqnarray}
where $\mathbf{1}$ denotes the indicator function. Let $S_N$ denote the set
$$
\left\{ \alpha \in [0,1]:~\left| \alpha - \frac{p}{q} \right| < \frac{1}{K_N^2 N^2} \text{ for some } 0 < q < \frac{N}{K_N},~(p,q)=1 \right\}.
$$
We have
$$
\textup{mes} (S_N) \gg \left(\frac{N}{K_N}\right)^2 \frac{1}{K_N^2 N^2} = \frac{1}{K_N^4},
$$
using well-known estimates for the average order of the Euler totient function (see for example~\cite{wal}). Also, for $\alpha \in S_N$ we clearly have the lower bound
\begin{equation} \label{b29}
\frac{1}{2K_N} \frac{K_N^2}{10} \gg K_N
\end{equation}
for the expression in line~\eqref{b26}. Now consider only indices $N$ along an extremely thin subsequence $(N_j)_{j \geq 1}$ of $\mathbb{N}$, and assume that $K_N$ increases so slowly with $N$ that 
\begin{equation} \label{b10}
\sum_{j=1}^\infty \frac{1}{K_{N_j}^4} = \infty. 
\end{equation}
The fast growth of $N_j$ allows us to consider the sets $S_{N_j}$ as being essentially independent.\footnote{The required ``almost independence'' property can be deduced from the fact that the Farey fractions are asymptotically equidistributed. Precise discrepancy estimates for the Farey fractions are known (see~\cite{dress,nied}), which could be used to obtain a quantitative version of this proof. The ``appropriate'' version of the Borel--Cantelli lemma mentioned in the next sentence could be for example the Erd\H os--R\'enyi version, see e.g.~\cite[p. 391]{renyi}.} Hence, by~\eqref{b10} and an appropriate version of the second Borel--Cantelli lemma, the limsup set
$$
S = \bigcap_{j_0 \geq 1} \bigcup_{j \geq j_0} S_{N_j}
$$
has full measure. Defining $(a(x))_{x \geq 1}$ as the infinite sequence whose elements are all the numbers contained in $\bigcup_{j \geq 1} A_{N_j}$, sorted in increasing order, it follows from~\eqref{b26} and~\eqref{b29} that 
$$
\limsup_{j \to \infty} R_2([-1,1],\alpha,N_j) = \infty
$$
for all $\alpha \in S$. Thus $(a(x))_{x \geq 1}$ fails to have the metric pair correlation property despite satisfying $E(A_N) = o(N^3)$, thereby giving a negative answer to the question in the second problem.


\end{document}